\documentclass[4pt]{article} 

\usepackage[utf8]{inputenc} 
\usepackage{geometry} 
\usepackage{graphicx} 

\usepackage[colorlinks,citecolor=red]{hyperref}
\usepackage{amsmath}
\usepackage{amsfonts}
\usepackage{dsfont}
\usepackage{mathrsfs}
\usepackage{latexsym}
\usepackage{graphicx}
\usepackage{amssymb}
\usepackage{float}
\usepackage{epsfig}
\usepackage{epstopdf}
\usepackage{color,xcolor}
\usepackage{booktabs} 
\usepackage{array} 
\usepackage{paralist} 
\usepackage{verbatim} 
\usepackage{subfig} 

\newtheorem{theorem}{Theorem}[section]
\newtheorem{lemma}[theorem]{Lemma}

\newtheorem{proposition}[theorem]{Proposition}

\newtheorem{problem}[theorem]{Problem}

\usepackage{fancyhdr} 
\usepackage{sectsty}
\allsectionsfont{\sffamily\mdseries\upshape} 

\usepackage[nottoc,notlof,notlot]{tocbibind} 
\usepackage[titles,subfigure]{tocloft} 

\title{Enumerations of Universal Cycles for $k$-Permutations}
\author{Zuling Chang$^{a,b}$, Jie Xue$^{a}$\thanks{Emails:~zuling\_chang@zzu.edu.cn (Z. Chang),~jie\_xue@126.com (J. Xue).}
\\
{\footnotesize $^a$ School of Mathematics and Statistics, Zhengzhou University, 450001 Zhengzhou, China}\\
{\footnotesize $^b$ State Key Laboratory of Information Security, Institute of Information Engineering,}\\
{\footnotesize Chinese Academy of Sciences, 100093 Beijing, China}
}
\date{} 

\begin{document}
\maketitle

\begin{abstract}
Universal cycle for $k$-permutations is a cyclic arrangement in which each $k$-permutation appears exactly once as $k$ consecutive elements. Enumeration problem of universal cycles for $k$-permutations is discussed and one new enumerating method is proposed in this paper. Accurate enumerating formulae are provided when $k=2,3$.

\bigskip

\end{abstract}

\section{Introduction}

Given a positive integer $n$, let $[n]=\{1,2,\ldots,n\}$. A {\it $k$-permutation} is an ordered arrangement of $k$ distinct elements in $[n]$, $1\leq k\leq n$.
Let $P_{n,k}$ be the set of all $k$-permutations of the $n$-set $[n]$. Obviously, $|P_{n,k}|=n!/(n-k)!$.
Let $C=(x_{1},x_{2},\ldots,x_{|P_{n,k}|})$ be a cyclic arrangement (or periodic sequence), where each $x_{i}\in [n]$ for $1\leq i\leq |P_{n,k}|$.
If in $C$ each $k$-permutation appears exactly once as $k$ consecutive elements, then we say that $C$ is a {\it universal cycle} for $P_{n,k}$.
For example, if $n=4$ and $k=2$, then $(123413242143)$ is a universal cycle for $P_{4,2}$, as follows.
\[
\begin{matrix}
  1 & 2 & 3 & 4\\
  3 &  &  & 1 \\
    4 &  &  & 3 \\
  1 & 2 & 4 & 2\\
\end{matrix}~~~~~~\leftrightarrows~~~~~~ P_{4,2}=\{12, 13, 14, 21, 23, 24, 31, 32, 34, 41, 42, 43\}
\]


Universal cycles were introduced by Chung, Diaconis, and Graham \cite{Chung1992} as generalizations of de Bruijn cycles \cite{Bruijn46}, which are binary sequences with period $2^n$ that contain every binary $n$-tuple.
Universal cycles are connected with Gray codes deeply \cite{Knuth2005,Sedgewick1977}. In this paper we consider the universal cycles for $k$-permutations. Jackson \cite{Jackson1993} showed that the universal cycle for $k$-permutations always exists when $k<n$. There are lots of results about the construction of universal cycles for $k$-permutations, mainly for the case that $k=n-1$ named {\it shorthand permutations} \cite{Gabric2020,Holroyd2012,Johnson2009,Ruskey2010}.
Another interesting problem is to compute the number of distinct universal cycles for $k$-permutations. This problem was formally presented in \cite{Jackson2009}.

\begin{problem}{\rm(Problem 477 \cite{Jackson2009})}
  How many different universal cycles for $P_{n,k}$ exist?
\end{problem}

Up to now, such enumeration problem has not been solved. When $k=1$, the number of universal cycles for $P_{n, 1}$ is obviously equal to $(n-1)!$. But for $k\geq 2$, counting them is a little complicated. In this paper, we propose one new method to count them and accurate formulae for the number of universal cycles are provided when $k=2$ and $3$. The following two theorems are main results of this paper.

\begin{theorem}\label{the1}
The number of universal cycles for $P_{n,2}$, $n\geq 3$, is equal to
\begin{equation}\label{equ1}
n^{n-2}[(n-2)!]^{n}.
\end{equation}
\end{theorem}

\begin{theorem}\label{the2}
The number of universal cycles for $P_{n,3}$, $n\geq 4$, is equal to
\begin{equation}\label{equ2}
(n-3)^{\frac{(n-1)(n-2)}{2}}(n-2)^{n-1}(n-1)^{\frac{(n-1)(n-2)}{2}-2}n^{n-2}[(n-3)!]^{n(n-1)}.
\end{equation}
\end{theorem}

The rest of this paper is presented as follows. In Section~\ref{sec:2} we collect preliminary notions and known results.
Related lemmas and detailed proofs of Theorems \ref{the1} and \ref{the2} are provided in Sections 3. We conclude this paper in the last section.

\section{Preliminaries}\label{sec:2}

Let us recall some definitions and concepts for digraphs.
For a vertex $v$ in a digraph, its out-degree is the number of arcs with initial vertex $v$, and its in-degree is the number of arcs with final vertex $v$.
A digraph is balanced if for each vertex, its in-degree and out-degree are the same.
It is well-known that a digraph contains an Eulerian tour if and only if the digraph is connected and balanced (see, for example, \cite[Theorem 1.7.2]{Bang2009}).

Given a digraph $D$, its adjacency matrix is the (0,1)-matrix $A=(a_{i,j})$
where $a_{i,j}=1$ if $v_{i}v_{j}$ is an arc of $D$, and $a_{i,j}=0$ otherwise.
Let $T$ be the diagonal matrix of the vertex out-degrees.
The Laplacian matrix of $D$ is defined as $L=T-A$.
The eigenvalues of $L$ are called the Laplacian eigenvalues of $D$.
Obviously, the row sum of $L$ is zero, which implies that $0$ is an eigenvalue of $L$ with respect to the eigenvector $\mathbf{1}$.

In order to count the number of distinct universal cycles for $P_{n,k}$, $k>1$, we define a transition digraph.
Let $D$ be a digraph with vertex set $P_{n,k-1}$.
The arcs of $D$ satisfy the following rule: for any two vertices $i_{1}i_{2}\cdots i_{k-1}$ and $j_{1}j_{2}\cdots j_{k-1}$,
there is an arc from $i_{1}i_{2}\cdots i_{k-1}$ to $j_{1}j_{2}\cdots j_{k-1}$ if and only if
$i_{s}=j_{s-1}$ for $2\leq s\leq k-1$, and $i_{1}\neq j_{k-1}$. Such a digraph is called the transition digraph of $P_{n,k}$.
Let $uv$ be an arc in $D$ with initial vertex $u$ and final vertex $v$.
If $u=i_{1}i_{2}\cdots i_{k-1}$, then $v=i_{2}i_{3}\cdots i_{k-1}i_{k}$, where $i_{k}\in [n]\backslash \{i_{1},i_{2},\ldots,i_{k-1}\}$,
and so the arc $uv$ may be regarded as the $k$-permutation $i_{1}i_{2}\cdots i_{k-1}i_{k}$.
On the other hand, any $k$-permutation $i_{1}i_{2}\cdots i_{k-1}i_{k}$ in $P_{n,k}$ is represented by an arc with initial vertex $i_{1}i_{2}\cdots i_{k-1}$
and final vertex $i_{2}i_{3}\cdots i_{k-1}i_{k}$. In \cite{Jackson1993}, Jackson showed such transition digraph is balanced and connected.  One can see that any Eulerian tour in this transition digraph corresponds to a universal cycle for $P_{n,k}$, which leads to the following proposition directly.

\begin{proposition}\label{prop1}
  The number of distinct universal cycles for $P_{n,k}$ is equal to the number of Eulerian tours of its transition digraph.
\end{proposition}

This proposition implies that it is sufficient to consider the number of Eulerian tours in the transition digraph of $P_{n,k}$.
In the following, we establish the formulae for the number of distinct universal cycles.
Let $D$ be a connected balanced digraph, and let $\epsilon(D)$ denote the number of Eulerian tours of $D$.
 We use $d^{+}(v)$ to denote the out-degree of the vertex $v$.

\begin{lemma}{\rm(\cite{Stanley2013})}\label{lem1}
  Let $D$ be a connected balanced digraph with vertex set $V$.
  If the Laplacian eigenvalues of $D$ are $\mu_{1}\geq \mu_{2}\geq \cdots\geq \mu_{|V|-1}>\mu_{|V|}=0$, then
\begin{equation}\label{equ:1}
  \epsilon(D)=\frac{1}{|V|}\mu_{1}\mu_{2}\cdots\mu_{|V|-1}\prod_{v\in V}(d^{+}(v)-1)!.
\end{equation}
\end{lemma}

According to Lemma \ref{lem1}, to count the number of universal cycles for $P_{n,k}$, it is enough to compute the corresponding Laplacian eigenvalues. In the following section, we will show how to determine them by a  simple method.

\section{Enumeration formulae for $k=2$ and $3$}\label{sec:3}

In this section we will count the number of universal cycles for $P_{n,2}$ and $P_{n,3}$, and the proofs of Theorems \ref{the1} and \ref{the2} are provided respectively.

Firstly we give the proof of Theorem \ref{the1} about the number of universal cycles for $P_{n,2}$.

\medskip
\noindent \textbf{Proof of Theorem \ref{the1}.} Let $D$ be the transition diagraph of $P_{n,2}$.
The vertex set of $D$ is $[n]$. For any two distinct vertices $i$ and $j$, by the definition of transition diagraph, $ij$ is an arc of $D$.
Note also that the out-degree of each vertex equals $n-1$. Thus, the Laplacian matrix of $D$ is
\begin{equation}\label{equ:2}
\begin{bmatrix}
  n-1 &-1 & -1 & \cdots & -1\\
  -1 & n-1 & -1 & \cdots & -1\\
  -1 & -1 & n-1 & \cdots & -1\\
  \vdots & \vdots & \vdots & \ddots & \vdots\\
  -1 & -1 & -1 & \cdots & n-1
\end{bmatrix}=nI-J,
\end{equation}
where $J$ is the $n\times n$ matrix of all ones, and $I$ is the $n\times n$ identity matrix. It is well known that the eigenvalues of $J$
are $n$ (once) and 0 ($n-1$ times) \cite{Stanley2013}. A simple calculation shows that the eigenvalues of this Laplacian matrix are
$$\underbrace{n,n,\ldots,n}_{n-1},0.$$
According to Lemma \ref{lem1}, we obtain that
$$\epsilon(D)=\frac{1}{n}n^{n-1}[(n-2)!]^{n}=n^{n-2}[(n-2)!]^{n}.$$
Thus, it follows from Proposition \ref{prop1} that there are $n^{n-2}[(n-2)!]^{n}$ distinct universal cycles for $P_{n,2}$.
\hspace*{\fill}$\Box$
\medskip

Next we consider the universal cycles for $P_{n,3}$.  Let $D$ be the transition diagraph of $P_{n,3}$ with adjacency matrix $A$.
One can see that $D$ is balanced, and the out-degree (in-degree) of any vertex equals $n-2$.
The following lemma presents the property of the adjacency matrix $A$.

\begin{lemma}\label{lem2}
  Let $D$ be the transition diagraph of $P_{n,3}$ with adjacency matrix $A$. Then
\begin{equation}\label{equ:3}
 A^{4}+A^{3}+(n-3)A^{2}-A-(n-2)I=(n-2)(n-3)J,
\end{equation}
where $J$ is the $n(n-1)\times n(n-1)$ matrix of all ones, and $I$ is the $n(n-1)\times n(n-1)$ identity matrix.
\end{lemma}

\begin{proof}
Let $\tau_{\ell}(u,v)$ denote the number of walks from $u$ to $v$ in $D$ of length $\ell\geq 1$.
In the following, we determine $\tau_{\ell}(u,v)$ by distinguishing seven cases for any ordered pair $(u,v)$.

\medskip

\noindent\textbf{Case 1:} $u=ab$ and $v=ab$, where $a,b\in [n]$.

\medskip

Obviously, $\tau_{1}(ab,ab)=\tau_{2}(ab,ab)=0$. The 3-walk from $ab$ to $ab$ may be indicated as
$$ab\longrightarrow bx\longrightarrow xa \longrightarrow ab$$
where $x\in [n]\backslash \{a,b\}$. Thus, $\tau_{3}(ab,ab)=n-2$. Similarly, the 4-walk from $ab$ to $ab$ is presented as
$$ab\longrightarrow bx\longrightarrow xy\longrightarrow ya \longrightarrow ab$$
where $x,y\in [n]\backslash \{a,b\}$ and $x\neq y$. Then $\tau_{4}(ab,ab)=(n-2)(n-3)$.

\medskip

\noindent\textbf{Case 2:} $u=ab$ and $v=ba$, where $a,b\in [n]$.

\medskip

It is easy to see that there is not walk from $ab$ to $ba$ of length $\ell\leq 3$, and so $\tau_{1}(ab,ba)=\tau_{2}(ab,ba)=\tau_{3}(ab,ba)=0$.
The 4-walk from $ab$ to $ba$ is expressed as
$$ab\longrightarrow bx\longrightarrow xy\longrightarrow yb \longrightarrow ba.$$
Thus, $x,y\in [n]\backslash \{a,b\}$ and $x\neq y$, which yielding $\tau_{4}(ab,ba)=(n-2)(n-3)$.

\medskip

\noindent\textbf{Case 3:} $u=ab$ and $v=ac$, where $a,b,c\in [n]$.

\medskip

There is not walk from $ab$ to $ac$ of length $\ell\leq 2$, hence $\tau_{1}(ab,ac)=\tau_{2}(ab,ac)=0$. The 3-walk from $ab$ to $ac$ may be presented as
$$ab\longrightarrow bx\longrightarrow xa \longrightarrow ac$$
where $x$ belongs to $[n]\backslash \{a,b,c\}$. Then $\tau_{3}(ab,ac)=n-3$. The 4-walk from $ab$ to $ac$ is
$$ab\longrightarrow bx\longrightarrow xy\longrightarrow ya \longrightarrow ac.$$
Obviously, $x\in [n]\backslash \{a,b\}$, $y\in [n]\backslash \{a,b,c\}$ and $x\neq y$. Thus, $\tau_{4}(ab,ac)=n-3+(n-3)(n-4)=(n-3)^{2}$.

\medskip

\noindent\textbf{Case 4:} $u=ab$ and $v=ca$, where $a,b,c\in [n]$.

\medskip

In this case, $uv$ is not an arc in $D$, thus $\tau_{1}(ab,ca)=0$. Moreover, $ab\longrightarrow bc\longrightarrow ca$ is the unique 2-walk from $ab$ to $ca$,
and so $\tau_{2}(ab,ca)=1$. Since the 3-walk from $ab$ to $ca$ is
$$ab\longrightarrow bx\longrightarrow xc \longrightarrow ca,$$
it means that $x\in [n]\backslash \{a,b,c\}$. Therefore, $\tau_{3}(ab,ca)=n-3$. Let
$$ab\longrightarrow bx\longrightarrow xy\longrightarrow yc \longrightarrow ca$$
be the 4-walk from $ab$ to $ca$. It follows that $x,y\in [n]\backslash \{a,b,c\}$ and $x\neq y$, hence $\tau_{4}(ab,ca)=(n-3)(n-4)$.

\medskip

\noindent\textbf{Case 5:} $u=ab$ and $v=bc$, where $a,b,c\in [n]$.

\medskip

Clearly, $uv$ is an arc in $D$, that is, $\tau_{1}(ab,bc)=1$. Since there is not walk from $ab$ to $bc$ of length 2 or 3, we have $\tau_{2}(ab,bc)=\tau_{3}(ab,bc)=0$.
The 4-walk from $ab$ to $bc$ is presented as
$$ab\longrightarrow bx\longrightarrow xy\longrightarrow yb \longrightarrow bc.$$
It follows that $x\in [n]\backslash \{a,b\}$, $y\in [n]\backslash \{b,c\}$ and $x\neq y$.
Therefore, $\tau_{4}(ab,bc)=(n-3)^{2}+(n-2)$.

\medskip

\noindent\textbf{Case 6:} $u=ab$ and $v=cb$, where $a,b,c\in [n]$.

\medskip

It is easy to see that the length of any walk from $ab$ to $cb$ is at least 3. Suppose that
$$ab\longrightarrow bx\longrightarrow xc \longrightarrow cb$$
is a 3-walk from $ab$ to $cb$. Thus, $x\in [n]\backslash \{a,b,c\}$, and so $\tau_{3}(ab,cb)=n-3$. Let
$$ab\longrightarrow bx\longrightarrow xy\longrightarrow yc \longrightarrow cb$$
be a 4-walk from $ab$ to $cb$. It follows that $x\in [n]\backslash \{a,b,c\}$, $y\in [n]\backslash \{b,c\}$ and $x\neq y$,
which implies that $\tau_{4}(ab,cb)=n-3+(n-3)(n-4)=(n-3)^{2}$.

\medskip

\noindent\textbf{Case 7:} $u=ab$ and $v=cd$, where $a,b,c,d\in [n]$.

\medskip

Since $uv$ is not an arc in $D$, $\tau_{1}(ab,cd)=0$. Note that $ab\longrightarrow bc\longrightarrow cd$ is the unique 2-walk from $ab$ to $cd$.
Hence $\tau_{2}(ab,cd)=1$. For any $x\in [n]\backslash \{a,b,c,d\}$, $ab\longrightarrow bx\longrightarrow xc \longrightarrow cd$ forms a 3-walk,
thus $\tau_{3}(ab,cd)=n-4$. The 4-walk from $ab$ to $cd$ is presented as
$$ab\longrightarrow bx\longrightarrow xy\longrightarrow yc \longrightarrow cd,$$
in which $x\in [n]\backslash \{a,b,c\}$, $y\in [n]\backslash \{b,c,d\}$ and $x\neq y$. It follows that $\tau_{4}(ab,cd)=n-3+(n-4)^{2}$.

\begin{table}[t]
\caption{Entries of $A, A^{2}, A^{3}$, $A^{4}$ and $A^{4}+A^{3}+(n-3)A^{2}-A$}\label{tab1}
\centering
\begin{tabular}{lccccc}
\toprule
 & $A$ & $A^{2}$ & $A^{3}$ & $A^{4}$ & $A^{4}+A^{3}+(n-3)A^{2}-A$ \\
\midrule
$(ab,ab)$-entry & 0 & 0 & $n-2$ & $(n-2)(n-3)$ & $(n-2)^{2}$\\
$(ab,ba)$-entry & 0 & 0 & 0 & $(n-2)(n-3)$ & $(n-2)(n-3)$\\
$(ab,ac)$-entry & 0 & 0 & $n-3$  & $(n-3)^{2}$ & $(n-2)(n-3)$\\
$(ab,ca)$-entry & 0 & 1 & $n-3$  & $(n-3)(n-4)$ & $(n-2)(n-3)$\\
$(ab,bc)$-entry & 1 & 0 & 0  & $(n-3)^{2}+(n-2)$ & $(n-2)(n-3)$\\
$(ab,cb)$-entry & 0 & 0 & $n-3$  & $(n-3)^{2}$ & $(n-2)(n-3)$\\
$(ab,cd)$-entry & 0 & 1 & $n-4$  & $(n-3)+(n-4)^{2}$ & $(n-2)(n-3)$\\
\bottomrule
\end{tabular}
\end{table}

Recall that the $(u,v)$-entry of the matrix $A^{\ell}$ is equal to $\tau_{\ell}(u,v)$.
According to the above cases, we determine the entries of matrices $A, A^{2}, A^{3}$ and $A^{4}$ respectively, as shown in Table \ref{tab1}.
Assume that the matrix $A$ satisfies the following equation
\[
\alpha_{0}I+\alpha_{1}A+\alpha_{2}A^{2}+\alpha_{3}A^{3}+\alpha_{4}A^{4}=J.
\]
By Table \ref{tab1}, we obtain the system of equations
\begin{equation*}
\left\{
             \begin{array}{lr}
             \alpha_{0}+\alpha_{3}(n-2)+\alpha_{4}(n-2)(n-3)=1,\\
             \alpha_{4}(n-2)(n-3)=1,\\
             \alpha_{3}(n-3)+\alpha_{4}(n-3)^{2}=1,\\
             \alpha_{2}+\alpha_{3}(n-3)+\alpha_{4}(n-3)(n-4)=1,\\
             \alpha_{1}+\alpha_{4}((n-3)^{2}+(n-2))=1,\\
             \alpha_{2}+\alpha_{3}(n-4)+\alpha_{4}((n-3)+(n-4)^{2})=1.
             \end{array}
\right.
\end{equation*}
The solution to this system of equations is given by
\begin{equation*}
\left\{
             \begin{array}{lr}
             \alpha_{0}=-\frac{1}{n-3},\\
            \alpha_{1}=-\frac{1}{(n-2)(n-3)},\\
             \alpha_{2}=\frac{1}{n-2},\\
             \alpha_{3}=\frac{1}{(n-2)(n-3)},\\
             \alpha_{4}=\frac{1}{(n-2)(n-3)}.
             \end{array}
\right.
\end{equation*}
Thus, it follows that
\[
-\frac{1}{n-3}I-\frac{1}{(n-2)(n-3)}A+\frac{1}{n-2}A^{2}+\frac{1}{(n-2)(n-3)}A^{3}+\frac{1}{(n-2)(n-3)}A^{4}=J,
\]
that is,
$$-(n-2)I-A+(n-3)A^{2}+A^{3}+A^{4}=(n-2)(n-3)J,$$
and Equation (\ref{equ:3}) has been deduced.
\end{proof}

\begin{lemma}\label{lem3}
Let $D$ be the transition diagraph of $P_{n,3}$ with adjacency matrix $A$. Then the eigenvalues of $A$ are
$$n-2,\underbrace{1,\ldots,1}_{\frac{(n-1)(n-2)}{2}},\underbrace{-1,\ldots,-1}_{\frac{n(n-3)}{2}},\underbrace{p\ldots,p}_{n-1},\underbrace{q\ldots,q}_{n-1},$$
where $p$ and $q$ are the roots of $\lambda^{2}+\lambda+n-2=0$.
\end{lemma}
\begin{proof}
Here we define a function $\phi(x)=x^{4}+x^{3}+(n-3)x^{2}-x-(n-2)$.
If the eigenvalues of $A$ are
$$\lambda_{1},\lambda_{2},\ldots,\lambda_{i},\ldots,\lambda_{n(n-1)},$$
then the eigenvalues of $A^{4}+A^{3}+(n-3)A^{2}-A-(n-2)I$ are
\begin{equation*}\label{eq1}
\phi(\lambda_{1}),\phi(\lambda_{2}),\ldots, \phi(\lambda_{i}),\ldots, \phi(\lambda_{n(n-1)}),
\end{equation*}
which will be the eigenvalues of $(n-2)(n-3)J$
\begin{equation*}\label{eq2}
  n(n-1)(n-2)(n-3),\underbrace{0,0,\ldots,0}_{n(n-1)-1},
\end{equation*}
according to Lemma \ref{lem2}.

Recall that the out-degree of any vertex in $D$ is equal to $n-2$, that is, the row sum of $A$ is $n-2$.
Let $\mathbf{1}$ be a column vector of all ones. It is obvious that $A\mathbf{1}=(n-2)\mathbf{1}$, which implies that $n-2$ is an eigenvalue of $A$.
Moreover, since $\phi(n-2)=n(n-1)(n-2)(n-3)$, the multiplicity of $n-2$ as an eigenvalue of $A$ is 1.
Therefore, the remaining eigenvalues of $A$ are the roots of $\phi(\lambda)=0$, i.e.,
$$\lambda^{4}+\lambda^{3}+(n-3)\lambda^{2}-\lambda-(n-2)=(\lambda-1)(\lambda+1)(\lambda^{2}+\lambda+n-2)=0.$$
So we can deduce that the characteristic polynomial of $A$ is
$$|\lambda I-A|=(\lambda-n+2)(\lambda-1)^{s_{1}}(\lambda+1)^{s_{2}}(\lambda^{2}+\lambda+n-2)^{s_{3}},$$
where $s_{1},s_{2},s_{3}\geq0$ and $s_1+s_2+2s_3=n(n-1)-1$. In other words, the eigenvalues of $A$ are
$$n-2,\underbrace{1,\ldots,1}_{s_{1}},\underbrace{-1,\ldots,-1}_{s_{2}},\underbrace{p\ldots,p}_{s_{3}},\underbrace{q\ldots,q}_{s_{3}},$$
where $p$ and $q$ are the roots of $\lambda^{2}+\lambda+n-2=0$. From Table \ref{tab1},
we obtain that the traces of $A, A^{2}$ and $A^{3}$ are $0,0$ and $n(n-1)(n-2)$, respectively.
Using the relationship between eigenvalues and trace of a matrix, we obtain
\begin{equation}\label{eq3}
\left\{
             \begin{array}{lr}
             n-2+s_{1}-s_{2}+s_{3}(p+q)=0,\\
             (n-2)^{2}+s_{1}+s_{2}+s_{3}(p^{2}+q^{2})=0,\\
             (n-2)^{3}+s_{1}-s_{2}+s_{3}(p^{3}+q^{3})=n(n-1)(n-2).
             \end{array}
\right.
\end{equation}
Since $p$ and $q$ are the roots of $\lambda^{2}+\lambda+n-2=0$, it follows that
\begin{equation}\label{eq4}
\left\{
             \begin{array}{l}
             p+q=-1,\\
             p^{2}+q^{2}=(p+q)^{2}-2pq=-2n+5,\\
             p^{3}+q^{3}=(p+q)(p^{2}+q^{2}-pq)=3n-7.
             \end{array}
\right.
\end{equation}
Combining (\ref{eq3}) and (\ref{eq4}), we obtain that $s_{1}=(n-1)(n-2)/2$, $s_{2}=n(n-3)/2$ and $s_{3}=n-1$, and the result follows.
\end{proof}

Now, we provide the proof of Theorem \ref{the2}.

\medskip
\noindent \textbf{Proof of Theorem \ref{the2}.} Let $D$ be the transition diagraph of $P_{n,3}$ with adjacency matrix $A$ and Laplacian matrix $L$. Since the out-degree of any vertex in $D$ is $(n-2)$,
we have $L=(n-2)I-A$. Thus, it follows from Lemma \ref{lem3} that the eigenvalues of $L$ are
$$0,\underbrace{n-3,\ldots,n-3}_{\frac{(n-1)(n-2)}{2}},\underbrace{n-1,\ldots,n-1}_{\frac{n(n-3)}{2}},\underbrace{n-2-p\ldots,n-2-p}_{n-1},
\underbrace{n-2-q\ldots,n-2-q}_{n-1},$$
where $p$ and $q$ are the roots of $\lambda^{2}+\lambda+n-2=0$. Since $pq=n-2$ and $p+q=-1$, we have
\[
(n-2-p)(n-2-q)=(n-2)^{2}-(p+q)(n-2)+pq=n(n-2).
\]
According to Lemma \ref{lem1}, we obtain that
\begin{eqnarray*}
\epsilon(D)&=&\frac{1}{n(n-1)}(n-3)^{\frac{(n-1)(n-2)}{2}}(n-1)^{\frac{n(n-3)}{2}}[n(n-2)]^{n-1}[(n-3)!]^{n(n-1)}\\
&=&(n-3)^{\frac{(n-1)(n-2)}{2}}(n-2)^{n-1}(n-1)^{\frac{(n-1)(n-2)}{2}-2}n^{n-2}[(n-3)!]^{n(n-1)}.
\end{eqnarray*}
The number of distinct universal cycles for $P_{n,3}$ can be deduced directly based on Proposition \ref{prop1}.\hspace*{\fill}$\Box$
\medskip

\section{Conclusions}

By discussing the entries of powers of adjacency matrix $A$, we construct a polynomial of $A$, and then the eigenvalues of $A$ can be provided. According to this new method, enumerating results of universal cycles for $P_{n,2}$ and $P_{n,3}$ are proposed precisely. This new method can be applied to general case by analyzing the entries of $A^i$.

\section*{Acknowledgements}
This work was supported by the National Natural Science Foundation of China (Nos. 61772476 and 12001498).


\end{document}